\pdfoutput=1
\documentclass[11pt,oneside,reqno,british]{amsart}
\usepackage[T1]{fontenc}
\usepackage[latin9]{inputenc}
\usepackage{color}
\usepackage{babel}
\usepackage{mathtools}
\usepackage{amstext}
\usepackage{amsthm}
\usepackage{amssymb}
\usepackage{cancel}
\usepackage[unicode=true,pdfusetitle,
 bookmarks=true,bookmarksnumbered=false,bookmarksopen=false,
 breaklinks=true,pdfborder={0 0 0},pdfborderstyle={},backref=false,colorlinks=true]
 {hyperref}

\makeatletter
\theoremstyle{plain}
\newtheorem*{thm*}{\protect\theoremname}
\theoremstyle{plain}
\newtheorem{thm}{\protect\theoremname}
\theoremstyle{plain}
\newtheorem{lem}[thm]{\protect\lemmaname}
\theoremstyle{remark}
\newtheorem{rem}[thm]{\protect\remarkname}

\usepackage{calrsfs}
\usepackage{graphicx}
\usepackage{color}
\usepackage{perpage}
\usepackage{upquote}
\usepackage{bbm}
\usepackage[shortlabels]{enumitem}
\usepackage{stmaryrd} 
\usepackage{extarrows} 
\usepackage{romanbar}
\usepackage{xparse}
\usepackage{hiddennote}
\usepackage{latexml}

\MakePerPage{footnote}
\gdef\SetFigFontNFSS#1#2#3#4#5{} 
\gdef\SetFigFont#1#2#3#4#5{} 
\def\clap#1{\hbox to 0pt{\hss#1\hss}}

\DeclareMathOperator{\dist}{dist}

\def\sym{\mathfrak{S}}

\definecolor{myblue}{rgb}{0.09,0.32,0.44} 
\hypersetup{pdfborder={0 0 0},pdfborderstyle={},colorlinks=true,linkcolor=myblue,citecolor=myblue,urlcolor=blue}

\theoremstyle{remark}
\newtheorem*{qst*}{Question}
\newtheorem*{rmrks*}{Remarks}

\newlength{\tempindent} 
\newcommand{\lazyenum}{
\setlength{\tempindent}{\parindent} 
\begin{enumerate}[leftmargin=0cm,itemindent=0.7cm,labelwidth=\itemindent,labelsep=0cm,align=left,label=\arabic*)]
\setlength{\parskip}{\smallskipamount}
\setlength{\parindent}{\tempindent}
}

\makeatletter
\def\moverlay{\mathpalette\mov@rlay}
\def\mov@rlay#1#2{\leavevmode\vtop{%
   \baselineskip\z@skip \lineskiplimit-\maxdimen
   \ialign{\hfil$\m@th#1##$\hfil\cr#2\crcr}}}
\newcommand{\charfusion}[3][\mathord]{
    #1{\ifx#1\mathop\vphantom{#2}\fi
        \mathpalette\mov@rlay{#2\cr#3}
      }
    \ifx#1\mathop\expandafter\displaylimits\fi}
\makeatother

\makeatletter
\ifdefined\andify
\renewcommand{\andify}{%
  \nxandlist{\unskip, }{\unskip{} \@@and~}{\unskip{} \@@and~}}
\def\author@andify{%
  \nxandlist {\unskip ,\penalty-1 \space\ignorespaces}%
    {\unskip {} \@@and~}%
    {\unskip \penalty-2 \space \@@and~}%
}
\let\@wraptoccontribs\wraptoccontribs
\@namedef{subjclassname@2020}{%
  \textup{2020} Mathematics Subject Classification}
\fi
\makeatother

\NewDocumentCommand\afs{vm}{\emph{\href{#1}{#2}}}

\iflatexml
\long\def\HNOC#1#2#3{#2}
\else
\long\def\HNOC#1#2#3{\HiddennoteLink{#1}{\OptionalColour{#2}}{#3}}
\fi

\makeatother

\providecommand{\lemmaname}{Lemma}
\providecommand{\remarkname}{Remark}
\providecommand{\theoremname}{Theorem}

\begin{document}
\title{Coupled but distant}
\author{Itai Benjamini and Gady Kozma}
\maketitle

\section{Introduction}

Coupling is one of the most useful ideas in probability in general,
and in the study of random walks specifically. The idea is, usually,
that if there exists a coupling that will make two random walkers
reach the same point at the same time with high probability, then
it can be shown that various properties of the random walk do not
depend strongly on the starting point. We skip further discussion
and refer to the books \cite{Lind,LP}.

Is there some use for coupling random walks so that they reach the
same point but not necessarily at the same time? An old idea, floating
since at least the early millenium, that has not found an impressive
application yet is as follows. Assume you prove somehow that random
walk and loop-erased random walk intersect in space. Then you can
couple two loop-erased random walks using Wilson's algorithm. But
since loop-erased random walks exit a sphere at the same place their
generating random walks exit the sphere, what Wilson's algorithm actually
gives you is a coupling of two random walks so that they exit the
sphere at the same point. For most applications, this is as good as
coupling them to get to the same place at the same time. 

Here we ask the opposite question: can one couple two random walkers
so that they do not intersect? Again, we mean by `intersect' getting
to the same spot but not necessarily at the same time (see \cite{AHMWW}
for intersections at the same time). It is certainly not possible
to couple two random walks so that they do not intersect with probability
1, as the first random walk has a positive probability to reach each
neighbour of the starting point of the second random walk, and then
no coupling can prevent an intersection. So the best one might hope
for is a coupling such that, with positive probability the two random
walks do not intersect.

Thus we have reached the following question. Take two random walkers
in $\mathbb{Z}^{d}$ starting from, say, two neighbouring sites. Is
there a coupling such that
\[
\mathbb{P}(R_{1}\cap R_{2}=\emptyset)>0?
\]
This is not possible to achieve when $d=2$, because, with probability
1, $R_{1}$ does a full rotation around the starting point of $R_{2}$
creating a barrier. It is trivial when $d\ge5$ because, as is well-known,
the property holds when $R_{1}$ and $R_{2}$ are independent. Thus
we are left with the cases that $d$ is 3 and $4$. We were not able
to resolve the 3 dimensional case. Thus our result is:
\begin{thm*}
In $\mathbb{Z}^{4}$ it is possible to couple two random walkers so
that they do not ever intersect with positive probability.
\end{thm*}
Currently we have no application for this theorem. 

To get a feeling for the proof, let us remark that we do not construct
the coupling explicitly, but rather use Hall's marriage theorem to
show a coupling exists. To justify the conditions of the theorem,
we need to show that `hittable walks are unlikely'. Namely, suppose
we call a path $\gamma$ hittable if $\mathbb{P}(R\cap\gamma\ne\emptyset)>1-\varepsilon$
for some parameter $\varepsilon>0$, where $R$ is random walk. Then
we need to show that $\mathbb{P}(R$ is hittable$)\ll\varepsilon$.
In a single formula, if $R_{1}$ and $R_{2}$ are two random walks
then we need to show that $\mathbb{P}(\mathbb{P}(R_{1}\cap R_{2}\ne\emptyset\,|\,R_{1})>1-\varepsilon)\ll\varepsilon$.
We prove this by claiming that if it were not so, then one could see
it in the \emph{expected }number of intersections of $R_{1}$ with
\emph{many }independent walks. We then estimate the probability of
this event using moment methods. The rest of the paper is devoted
to the details. But first a few questions.

\begin{qst*}Is it possible to construct such a coupling in a Markovian
way, i.e.\ such that the process $(R_{1},R_{2})$ is a Markov process?
\end{qst*}

\begin{qst*}Is there any transient graph on which it is not possible
to construct such a coupling?\end{qst*}

One way to solve this last question would be to find a transient graph
on which random walk forms a cutset with probability 1 (see also \cite{BGGS}
for a related, but different result). We could not even disprove the
existence of such a graph.

\begin{qst*}Is there any transient graph on which the trace of random
walk contains a cutset separating the starting point from infinity
with probability 1?\end{qst*}

\subsection{Reading this paper}

\iflatexml\else\SetOptionalColour{red}\fi

The default format for this paper has several places that can be pressed
to open up an extra explanation page. Currently they are not marked
especially (we find too much colour distracting, and bad for printing).
To see them, press \iflatexml here \else\SetDefaultColour{green}\ToggleOptionalColour{here}\SetDefaultColour{black} \fi
and they will appear in red. 

This does not work in all pdf readers --- we tested this in Acrobat
Reader and in Evince (a Linux pdf browser) and it works in them. Most
other browsers we tried simply ignore these, so you can still read
the paper, just will not have access to these extra explanations pages. 

If these explanation pages give you problems with your pdf reader,
you can get a version that does not have this feature on by the following
procedure. Go to the arXiv page of this paper. Click on `TeX Source'
on the menu on the right-hand side. You will get a zipped tar file
--- open it and extract the file `avoiding coupling no yellow boxes.pdf'. 

\subsection{Notation}

All random walks will be on $\mathbb{Z}^{4}$, and they all start
at time zero (i.e.\ when we say `$R$ is random walk starting from
$x$' we mean $R(0)=x$). For two times $t<s$, the notation $R[t,s]$
will mean the path $(R(t),\dotsc,R(s))$. When we apply set notation
to paths (in particular to random walks), we always implicitly imply
the trace of the path, i.e.\ the set of vertices visited by the path.
Thus $R_{1}\cap R_{2}$ is the set of vertices of $\mathbb{Z}^{4}$
visited by both $R_{1}$ and $R_{2}$, and $|R_{1}\cap R_{2}|$ is
its size. We often say about a random walk $R$ that it is `stopped'
on some stopping time (typically upon hitting some subset of $\mathbb{Z}^{4}$).
In this case the notation $R$ refers only to the set of vertices
visited up to that stopping time. We denote by $G(x,y)$ the probability
that random walk starting from $x$ ever visits $y$ (in particular,
$G(x,x)=1$). 

Other random walk preliminaries are in the appendix.

We denote by $B(n)$ the discrete ball around 0 of radius $n$, namely
$\{x:|x|<n\}$, where $|x|$ is the $\ell^{2}$ (Euclidean) norm.
If we want a ball around another point we will use $B(x,n)\coloneqq\{y:|x-y|<n\}$.
We will use $\partial$ for the inner vertex boundary, namely $\partial A=\{x\in A:\exists y\not\in A,x\sim y\}$.
The notation $x\sim y$ is the neighbourhood relation in our graph,
namely $|x-y|=1$. 

As usual, $C$ and $c$ will refer to arbitrary universal constants,
whose value may change from place to place (even in the same formula).
$C$ will refer to constants which are `large enough' and $c$ to
constants which are `small enough'. We use $x\approx y$ as a short
for $cx\le y\le Cx$. All logs are to basis $2$. 

\section{Proof}

We will not construct the coupling explicitly but rather use the marriage
lemma to show one exists. For this we will need some estimates that
say that with high probability, a 4-dimensional random walk is not
hittable by a second random walk. In a formula, that $\mathbb{P}(\mathbb{P}(R_{1}\cap R_{2}\ne\emptyset\,|\,R_{1})>1-\varepsilon)$
is small, when $R_{1}$ and $R_{2}$ are two independent 4-dimensional
random walks (see lemma \ref{lem:hard to be hittable} below for the
precise formulation). The proof is a version of Lawler's result that
$\mathbb{P}(R_{1}\cap R_{2}\ne\emptyset)$ is small. 
\begin{lem}
\label{lem:adding one vertex}Let $x_{1},\dotsc,x_{r}\in\mathbb{Z}^{4}$
and let $E$ and $E'$ be the events that random walk starting from
0 passes through all of $x_{1},\dotsc,x_{r-1}$ and through all of
$x_{1},\dotsc,x_{r}$, respectively. Then
\[
\mathbb{P}(E')\le C(r)\mathbb{P}(E)(\dist(x_{r},\{0,x_{1},\dotsc,x_{r-1}\})+1)^{-2}.
\]
 
\end{lem}

\begin{proof}
Let $\sigma\in\sym_{r}$ i.e.\ a permutation of $\{1,\dotsc,r\}$.
Recall that $G(x,y)=\mathbb{P}^{x}(\exists t:R(t)=y)$ and denote
$x_{0}=0$, $\sigma(0)=0$, $\sigma(r+1)=\sigma(r-1)$, $k=\sigma^{-1}(r)-1$
and $d=\dist(x_{r},\{0,x_{1},\dotsc,x_{r-1}\})+1$ for brevity. Then
\begin{align*}
\mathbb{P}(E') & \le\sum_{\sigma\in\sym_{r}}\prod_{i=1}^{r}G(x_{\sigma(i-1)},x_{\sigma(i)})\\
 & \stackrel{(*)}{\le}Cd^{-2}\sum_{\sigma\in\sym_{r}}\Big(\prod_{i=1}^{k}G(x_{\sigma(i-1)},x_{\sigma(i)})\Big)G(x_{\sigma(k)},x_{\sigma(k+2)})\\
 & \qquad\qquad\Big(\prod_{i=k+3}^{r}G(x_{\sigma(i-1)},x_{\sigma(i)})\Big)\\
 & \le Cd^{-2}r!\mathbb{P}(E).
\end{align*}
To see $(*)$ assume first that $k\le r-2$ and that $|x_{r}-x_{\sigma(k)}|\le|x_{r}-x_{\sigma(k+2)}|$.
\HNOC{10cm}{In this case}{Write 
\[
x_{\sigma(k+2)}-x_{\sigma(k)}=x_{\sigma(k+2)}-x_{r}+x_{r}-x_{\sigma(k)}.
\]
Apply the triangle inequality and then the assumption $|x_{r}-x_{\sigma(k)}|\le|x_{r}-x_{\sigma(k+2)}|$
and get
\[
|x_{\sigma(k+2)}-x_{\sigma(k)}|\le|x_{\sigma(k+2)}-x_{r}|+|x_{r}-x_{\sigma(k)}|\le2|x_{r}-x_{\sigma(k+2)}|.
\]
Divide by 2 and get $|x_{r}-x_{\sigma(k+2)}|\ge\frac{1}{2}|x_{\sigma(k)}-x_{\sigma(k+2)}|$,
as claimed} $|x_{r}-x_{\sigma(k+2)}|\ge\frac{1}{2}|x_{\sigma(k)}-x_{\sigma(k+2)}|$
and hence by lemma \ref{lem:heat kernel}
\[
G(x_{r},x_{\sigma(k+2)})\le CG(x_{\sigma(k)},x_{\sigma(k+2)}).
\]
Together with $G(x_{\sigma(k)},x_{r})\le Cd^{-2}$, $(*)$ is proved
in this case. The case that $|x_{r}-x_{\sigma(k)}|>|x_{r}-x_{\sigma(k+2)}|$
is identical, with the roles of $\sigma(k)$ and $\sigma(k+2)$ interchanged.
The case that $k=r-1$ is simpler. The lemma is thus proved.
\end{proof}
\begin{lem}
\label{lem:the moments that killed me}Let $R_{0},R_{1},\dotsc,R_{k}$
be independent random walks in $\mathbb{Z}^{4}$ starting from $\partial B(n)$
and stopped when hitting $\partial B(m)$, for some $m>n$. Assume
further that their starting points satisfy $|R_{0}(0)-R_{i}(0)|>\frac{n}{\log n}$.
Then
\[
\mathbb{E}\bigg(\sum_{i=1}^{k}|R_{0}\cap R_{i}|\bigg)^{r}\le C(r)k^{3/2}\log\Big(\frac{m\log n}{n}\Big)\Big(\log m+k^{2}\log\Big(\frac{m\log n}{n}\Big)\Big)^{r-1}
\]
where $R_{0}\cap R_{i}$ denotes the intersection of the traces of
the two walks.
\end{lem}

We remark that both the $k^{3/2}$ and the $k^{2}$ can probably be
improved to $k$, but proving this weaker estimate is a little easier.
\begin{proof}
Throughout the proof, when we write $C$ or $c$ we mean a constant
which might depend on $r$, and we will not denote it by $C(r)$ the
way we did in the statement of the lemma. 

We first write
\begin{align}
\mathbb{E}\bigg(\sum_{i=1}^{k}|R_{0}\cap R_{i}|\bigg)^{r} & =\sum_{i_{1},\dotsc,i_{r}}\mathbb{E}\prod_{j=1}^{r}|R_{0}\cap R_{i_{j}}|\nonumber \\
 & =\sum_{i_{1},\dotsc,i_{r}}\sum_{x_{1},\dotsc,x_{r}}\mathbb{P}(x_{j}\in R_{0}\cap R_{i_{j}}\forall j)\label{eq:sum xi}
\end{align}
where the last sum is over all tuples of $x_{j}\in B(m)$. We fix
one such tuple $x_{j}$ and consider it.

Let $T$ be the minimal spanning tree over the $x_{j}$ and over $0$
i.e.\ $T$ is a subset of $\{0,\dotsc,r\}^{2}$ such that the graph
whose vertex set is $\{0,\dotsc,r\}$ and whose edge set is $T$ is
a spanning tree, and such that $\sum_{\{v,w\}\in T}|x_{v}-x_{w}|$
is minimised among all possible choices of $T$, where we denote $x_{0}\coloneqq0$.
If more than one choice exists, choose between them arbitrarily. For
any edge $\{v,w\}\in T$ either the path in $T$ from $0$ to $v$
passes through $w$ (in which case we say that $v$ is \emph{further})
or not (in which case we say that $w$ is further). For a $\{v,w\}\in T$
such that $v$ is further we denote 
\[
d_{\{v,w\}}\coloneqq|x_{v}-x_{w}|+1\qquad s_{\{v,w\}}\coloneqq\min_{i\in\{1,\dotsc,k\}}|x_{v}-R_{i}(0)|+1,
\]
i.e.\ $s_{\{v,w\}}$ is the distance of $x_{v}$ from the starting
points of $R_{1},\dotsc,R_{k}$ ($s_{\{v,w\}}$ depends also on the
tree $T$, as $T$ determines which of $v$ and $w$ is further, but
we omit this fact from the notation for brevity). Let $f$ be an edge
of $T$ containing 0 (if more than one exists, choose one arbitrarily).
Our first claim is that
\begin{equation}
\sum_{i_{1},\dotsc,i_{r}}\mathbb{P}(x_{u}\in R_{0}\cap R_{i_{u}}\forall u)\le Ckd_{f}^{-2}s_{f}^{-2}\prod_{e\in T\setminus\{f\}}d_{e}^{-4}+kd_{e}^{-2}s_{e}^{-2}.\label{eq:prod over tree}
\end{equation}
We will prove (\ref{eq:prod over tree}) by induction over $r$.

For the induction base we note that $\mathbb{P}(x\in R_{0}\cap R_{i})\le C|x|^{-2}|x-R_{i}(0)|^{-2}$,
by lemma \ref{lem:heat kernel}. Summing over $i$ gives (\ref{eq:prod over tree})
in the case $r=1$. For the induction step we \HNOC{10cm}{note}{Let
$T$ be an arbitrary finite tree and let $E$ be the number of edges
of $T$ and $V$ the number of vertices. Then $V=E+1$, which can
be proved by an induction on the number of vertices (we skip the details
of this proof, which is easy). `Non-trivial' here means that it is
not a single vertex, and then all degrees of all vertices are at least
1. Denote by $L$ the number of leaves, i.e.\ vertices with degree
1. Then the other vertices have degree at least 2 and the sum of degrees
(over all vertices) is at least $2(V-L)+L=2V-L$. Since the sum of
degrees is exactly $2E$ we get
\[
2V-L\le2E=2V-2
\]
or $L\ge2$.

As for the claim that a tree containing exactly 2 leaves is a path,
this follows from the same argument, but noting that $L=2$ requires
that the sum of degrees be equal to $2V-L$, which, in turn, requires
all degrees different from 1 to be 2.} that every non-trivial tree
contains at least two leaves and that a tree containing exactly two
leaves is a path. Let therefore $v\in\{1,\dotsc,r\}$ be a leaf of
$T$ different from 0 and such that $f\ne\{0,v\}$. Since $v$ is
a leaf there is exactly one $w$ such that $\{v,w\}\in T$ and $v$
is the further vertex. \HNOC{10cm}{Since}{Suppose there was some
$u\not\in\{v,w\}$ such that $|x_{u}-x_{v}|<|x_{w}-x_{v}|$. Then
we could construct a second graph $T'$ by taking $T$, removing the
edge $\{w,v\}$ and adding the edge $\{u,v\}$. The graph $T'$ would
still be connected (because we just disconnected and reconnected the
leaf $v$) and would still be a tree because any hypothetical cycle
in $T'$ cannot contain $v$ (because it has degree 1) but must contain
$v$ (because otherwise it would exist already in $T$). Hence $T'$
is a tree. Further, 
\[
\sum_{\{a,b\}\in T'}|x_{a}-x_{b}|-\sum_{\{a,b\}\in T}|x_{a}-x_{b}|=|x_{u}-x_{v}|-|x_{w}-x_{v}|<0,
\]
contradicting that $T$ is minimal with respect to this sum. Since
we reached a contradiction, we conclude that $|x_{u}-x_{v}|\ge|x_{w}-x_{v}|$
for all $u\not\in\{v,w\}$. } $T$ is a minimal spanning tree we
have
\begin{equation}
d_{\{v,w\}}=|x_{v}-w_{w}|+1\le|x_{v}-x_{u}|+1\quad\forall u\in\{0,\dotsc,r\}.\label{eq:dvw minimal}
\end{equation}
Use the induction hypothesis for $x_{1},\dotsc,\cancel{x_{v}},\dotsc,x_{r}$.
The tree $T\setminus\{\{v,w\}\}$ is a minimal spanning tree for $x_{1},\dotsc,\cancel{x_{v}},\dotsc,x_{r}$
and $f$ is an edge of $T\setminus\{\{v,w\}\}$ containing 0. In addition,
the notion of `further' in $T\setminus\{\{v,w\}\}$ is identical to
that notion in $T$, restricted. Hence the induction hypothesis says
that 
\[
\sum_{i_{1},\dotsc,\cancel{i_{v}},\dotsc i_{r}}\mathbb{P}(x_{u}\in R_{0}\cap R_{i_{u}}\forall u\ne v)\le Ckd_{f}^{-2}s_{f}^{-2}\prod_{e\in T\setminus\{\{v,w\},f\}}d_{e}^{-4}+kd_{e}^{-2}s_{e}^{-2}.
\]
By lemma \ref{lem:adding one vertex}, \HNOC{12.5cm}{the effect}{Recall
the statement of lemma \ref{lem:adding one vertex}:

\emph{Let $x_{1},\dotsc,x_{r}\in\mathbb{Z}^{4}$ and let $E$ and
$E'$ be the events that random walk starting from 0 passes through
all of $x_{1},\dotsc,x_{r-1}$ and through all of $x_{1},\dotsc,x_{r}$,
respectively. Then
\[
\mathbb{P}(E')\le C(r)\mathbb{P}(E)(\dist(x_{r},\{0,x_{1},\dotsc,x_{r-1}\})+1)^{-2}.
\]
}

\medskip 

The lemma is used as follows. Consider
\[
\sum_{i_{1},\dotsc,\cancel{i_{v}},\dotsc i_{r}}\mathbb{P}(\{x_{u}\in R_{0}\;\forall u\}\cap\{x_{u}\in R_{i_{u}}\;\forall u\ne v\}),
\]
i.e.\ the sum is still without $i_{v}$ and the event has changed
only for $R_{0}$, namely we require that all $x_{u}$ are in $R_{0}$,
including $x_{v}$, but do not make the requirement that $x_{v}\in R_{i_{v}}$.
Then lemma \ref{lem:adding one vertex}, applied to each summand individually
and only to $R_{0}$ gives
\[
\mathbb{P}(x_{u}\in R_{0}\;\forall u)\le C\mathbb{P}(x_{u}\in R_{0}\;\forall u\ne v)\cdot(\min_{u\ne v}|x_{v}-x_{u}|+1)^{-2}
\]
(as usual, $C$ is allowed to depend on $r$). Since $R_{0}$ and
$(R_{1},\dotsc,R_{k})$ are independent we have
\[
\mathbb{P}(\{x_{u}\in R_{0}\;\forall u\}\cap\{x_{u}\in R_{i_{u}}\;\forall u\ne v\})=\mathbb{P}(x_{u}\in R_{0}\;\forall u)\mathbb{P}(x_{u}\in R_{i_{u}}\;\forall u\ne v),
\]
and similarly if we do not require $x_{v}\in R_{0}$. Thus, 
\begin{multline*}
\mathbb{P}(\{x_{u}\in R_{0}\;\forall u\}\cap\{x_{u}\in R_{i_{u}}\;\forall u\ne v\})\le\\
C\mathbb{P}(x_{u}\in R_{0}\cap R_{i_{u}}\;\forall u\ne v)\cdot(\min_{u\ne v}|x_{v}-x_{u}|+1)^{-2}
\end{multline*}
Summing this over $i_{1},\dotsc,\linebreak[0]\cancel{i_{v}},\linebreak[0]\dotsc i_{r}$
gives
\begin{multline*}
\sum_{i_{1},\dotsc,\cancel{i_{v}},\dotsc i_{r}}\mathbb{P}(\{x_{u}\in R_{0}\;\forall u\}\cap\{x_{u}\in R_{i_{u}}\;\forall u\ne v\})\\
\le C(\min_{u\ne v}|x_{v}-x_{u}|+1)^{-2}\sum_{i_{1},\dotsc,\cancel{i_{v}},\dotsc i_{r}}\mathbb{P}(x_{u}\in R_{0}\cap R_{i_{u}}\forall u\ne v).
\end{multline*}
} of adding the condition $x_{v}\in R_{0}$ is bounded by multiplication
by $C(\min_{u}|x_{v}-x_{u}|+1)^{-2}$ which, by (\ref{eq:dvw minimal}),
is $Cd_{\{v,w\}}^{-2}$. We next consider the effect of adding the
summation over $i_{v}$ and the condition $x_{v}\in R_{i_{v}}$. We
need to consider two cases. The first is that $i_{v}$ is identical
to $i_{u}$ for some other $u\ne v$. In this case $i_{v}$ has only
$r$ possibilities. Denote by $U\subseteq\{1,\dotsc,r\}$ the set
$\{u:i_{v}=i_{u}\}$. Again, using lemma \ref{lem:adding one vertex}
gives that the effect of adding the condition $x_{v}\in R_{i_{j}}$
is bounded by 
\[
C\Big(\min\Big\{\min_{u\in U}|x_{v}-x_{u}|,|x_{v}-R_{i_{v}}(0)|\Big\}+1\Big)^{-2}\le C\min\{d_{\{v,w\}},s_{\{v,w\}}\}^{-2}.
\]
The second case is that $i_{v}$ is different from all existing $i_{u}$.
In this case the number of possibilities for $i_{u}$ is $k-r$ (which
could be much larger than $r$) but the effect of adding the condition
$x_{v}\in R_{i_{v}}$ is multiplying by a term $\approx|x_{v}-R_{i_{v}}(0)|^{-2}\le Cs_{\{v,w\}}^{-2}$.
All in all we get
\begin{align*}
\sum_{i_{1},\dotsc,i_{r}}\mathbb{P}(x_{u} & \in R_{0}\cap R_{i_{u}}\forall u)\le Ckd_{f}^{-2}s_{f}^{-2}\Big(\prod_{e\in T\setminus\{\{v,w\},f\}}d_{e}^{-4}+kd_{e}^{-2}s_{e}^{-2}\Big)\cdot\\
 & \qquad Cd_{\{v,w\}}^{-2}\cdot\Big(Cr\min\{d_{\{v,w\}},s_{\{v,w\}}\}^{-2}+C(k-r)s_{\{v,w\}}^{-2}\Big)\\
 & \le Ckd_{f}^{-2}s_{f}^{-2}\prod_{e\in T\setminus\{f\}}d_{e}^{-4}+kd_{e}^{-2}s_{e}^{-2},
\end{align*}
completing the induction and proving (\ref{eq:prod over tree}).

We now wish to sum (\ref{eq:prod over tree}) over all $x_{1},\dotsc,x_{r}$.
Fix one tree $T$ on $\{0,\dotsc,r\}$ and one edge $f$ of $T$ with
$0\in f$. Let $\mathcal{X}_{T}$ be the set of all tuples $(x_{1},\dotsc,x_{r})$
such that $T$ is their minimal spanning tree. Further, for $\alpha\in\{0,\dotsc,\lfloor\log_{2}m\rfloor\}$
let $\mathcal{X}_{T,u}$ be the set of all $(x_{1},\dotsc,x_{r})\in\mathcal{X}_{T}$
such that 
\[
\min_{e\in T}s_{e}\in\big[2^{\alpha},2^{\alpha+1}\Big).
\]
We now claim that, for any $T$, and $f$ and any $u$, 
\begin{multline}
\sum_{x\in\mathcal{X}_{T,\alpha}}kd_{f}^{-2}s_{f}^{-2}\prod_{e\in T\setminus\{f\}}d_{e}^{-4}+kd_{e}^{-2}s_{e}^{-2}\\
\le C\min\Big\{ k^{3/2},k^{2}\frac{4^{\alpha}\log^{2}n}{n^{2}}\Big\}\Big(\log m+k^{2}\log\Big(\frac{m}{2^{\alpha}}\Big)\Big)^{r-1}.\label{eq:sum over X}
\end{multline}
Again we prove (\ref{eq:sum over X}) by induction over $r$. The
case $r=1$ follows by noting that if $s_{f}\in[2^{\alpha},2^{\alpha+1})$
and $d_{f}\in[2^{\beta},2^{\beta+1})$ and then necessarily $\max\{\alpha,\beta\}\ge\log_{2}n/4\log n$
by our restriction on the starting points of the $R_{i}$. On the
other hand the number of $x$ satisfying these conditions is $\le C\min\{k2^{4\alpha},2^{4\beta}\}$.
We get
\begin{align}
\sum_{x\in\mathcal{X}_{\{\{0,1\}\},\alpha}}d_{f}^{-2}s_{f}^{-2} & \le C\quad\sum_{\mathclap{{\substack{\beta\in\{0,\dotsc,\log m\}\\
\max\{\alpha,\beta\}\ge\log(n/4\log n)
}
}}}\quad4^{-\alpha-\beta}\min\{k2^{4\alpha},2^{4\beta}\}\nonumber \\
 & \le C\min\Big\{\sqrt{k},k\frac{4^{\alpha}\log^{2}n}{n^{2}}\Big\}.\label{eq:sum dfsf}
\end{align}
(The second inequality follows by noting that the terms increase exponentially
in $\beta$ until there is an approximate equality in the minimum,
i.e.\ until $k2^{4\alpha}\approx2^{4\beta}$, and then decrease exponentially,
so that the sum is dominated by one term. When $\alpha<\log(n/4\log n)$
we get that the sum starts from $\beta=\lceil\log n/4\log n\rceil$,
but the sum is still dominated by a single term. The only difference
is that if, in addition, $2^{\alpha}<k^{-1/4}n/4\log n$ then the
maximal term is the one where $\beta=\lceil\log n/4\log n\rceil$,
but it does not satisfy $k2^{4\alpha}\approx2^{4\beta}$.) This shows
the case $r=1$ of (\ref{eq:sum over X}).

For the induction step we again choose a leaf $v\ne0$ of $T$ such
that $f\ne\{0,v\}$ and let $\{v,w\}\in T$. Fix $x_{1},\dotsc,\cancel{x_{v}},\dotsc,x_{r}$
(in particular fix $x_{w}$, so that $d_{\{v,w\}}$ is a function
of $x_{v}$ only). Define $\alpha'$ to be the corresponding quantity
for $T\setminus\{\{v,w\}\}$, i.e.\ by 
\[
\min_{e\in T\setminus\{\{v,w\}\}}s_{e}\in\big[2^{\alpha'},2^{\alpha'+1}\Big),
\]
so that $\alpha'\ge\alpha$. Consider first the case that $\alpha=\alpha'$
or $\alpha=\alpha'-1$. In this case we estimate 
\[
\sum_{x_{v}\in B(m)}d_{\{v,w\}}^{-4}\le C\log(m)
\]
because for each $l\in\{0,\dotsc,\lfloor\log_{2}m\rfloor\}$ there
are $\le C2^{4l}$ possibilities for $x_{v}$ such that $d_{\{v,w\}}\in[2^{l},2^{l+1}]$.
For the sum over $d_{\{v,w\}}^{-2}s_{\{v,w\}}^{-2}$ we perform a
calculation similar to (\ref{eq:sum dfsf}). Define $\beta$ and
$\gamma$ using $d_{\{v,w\}}\in[2^{\beta},2^{\beta+1})$ and $s_{\{v,w\}}\in[2^{\gamma},2^{\gamma+1})$,
and get 
\[
\qquad\sum_{\mathclap{x_{v}:x\in\mathcal{X}_{T,\alpha}}}d_{\{v,w\}}^{-2}s_{\{v,w\}}^{-2}\le C\sum_{\beta=0}^{\log m}\sum_{\gamma=\alpha}^{\log m}4^{-\beta-\gamma}\min\{2^{4\beta},k2^{4\gamma}\}\le C\sqrt{k}\log(m/2^{\alpha}).
\]
(Here and below $x=(x_{1},\dotsc,x_{r})$). Thus (\ref{eq:sum over X})
is proved in this case. 

The case that $\alpha<\alpha'-1$ is similar. We first note that this
means that $s_{\{v,w\}}\in[2^{\alpha},2^{\alpha+1})$ i.e.\ $v$
is close to the starting points. Thus it must be far from $w$, whose
distance to the starting points is at least $2^{\alpha'}$. Thus $\beta\ge\alpha'-1$
and the number of possibilities for $x_{v}$ is $\le Ck2^{4\alpha}$.
This makes both sums smaller. Indeed,
\[
\sum_{x_{v}:x\in\mathcal{X}_{T,\alpha}}d_{\{v,w\}}^{-4}\le\sum_{\beta=\alpha'}^{\log_{2}m}2^{-4\beta}\cdot Ck2^{4\alpha}\le Ck2^{4(\alpha-\alpha'}).
\]
and
\[
\sum_{x_{v}:x\in\mathcal{X}_{T,\alpha}}d_{\{v,w\}}^{-2}s_{\{v,w\}}^{-2}\le\sum_{\beta=\alpha'-1}^{\infty}Ck2^{4\alpha}\cdot4^{-\beta-\alpha}\le Ck4^{\alpha-\alpha'}
\]
Taking the term $4^{\alpha-\alpha'}$ out of the sum we again get
(\ref{eq:sum over X}) for $r$. This shows that (\ref{eq:sum over X})
is preserved when moving from $r$ to $r+1$, and hence holds for
all $r$.

Summing (\ref{eq:sum over X}) over $\alpha$ gives
\begin{align*}
\lefteqn{{\sum_{x\in\mathcal{X}_{T}}kd_{f}^{-2}s_{f}^{-2}\prod_{e\in T\setminus\{f\}}d_{e}^{-4}+kd_{e}^{-2}s_{e}^{-2}}}\qquad\qquad\\
 & \le\sum_{\alpha=0}^{\log m}C\min\Big\{ k^{3/2},k^{2}\frac{4^{\alpha}\log^{2}n}{n^{2}}\Big\}\Big(\log m+k^{2}\log\Big(\frac{m}{2^{\alpha}}\Big)\Big)^{r-1}.\\
 & \le Ck^{3/2}\log\Big(\frac{m\log n}{n}\Big)\Big(\log m+k^{2}\log\Big(\frac{m\log n}{n}\Big)\Big)^{r-1}.
\end{align*}
The lemma follows by summing over all trees and all possibilities
for the edge $f$.
\end{proof}
\begin{lem}
\label{lem:sraksrak}Let $R$ be random walk on $\mathbb{Z}^{4}$
starting from $0$ and let $n>0$. Then 
\[
\mathbb{P}\Big(\sum_{t=1}^{n}(|R(t)|+1)^{-2}>K\log n\Big)\le Cn^{-2}
\]
if only $K$ is sufficiently large.
\end{lem}

\begin{proof}
Let $\lambda,M>0$ be two parameters to be fixed later. Define stopping
times $\tau_{1}<\tau_{2}<\dotsb$ using $\tau_{1}=1$ and\begin{align*}
\sigma_i&\coloneqq\min\{t>\tau_{i}:|R(t)|>2|R(\tau_{i})|\text{ or }|R(t)|\HNOC{9cm}{\HNtemporaryhack{<\;}{<}}{The multiplication by $\mathbbm{1}_{|R(\tau_{i})|>M}$ simply means
that if $|R(\tau_{i})|\le M$ then this condition can never be fulfilled,
as it translates to $|R(t)|<0$ which cannot happen. Thus, for the
first few steps the process simply stops only when $|R(\tau_{i})|>2|R(\tau_{i})|$.
Only after $|R(\tau_{i})|$ rises above $M$ do we get a `real' random
walk on scales that can either go to lower or to higher scale. }(\tfrac{1}{2}|R(\tau_{i})|+1)\mathbbm{1}_{|R(\tau_{i})|>M}\}\\
\tau_{i+1}&\coloneqq\min\{\sigma_i,\tau_{i}+\lambda|R(\tau_{i})|^{2}\}.
\end{align*}
\HNOC{10cm}{By lemma}{Recall the statement of lemma \ref{lem:exit annulus}:

\emph{Let $a<1<A$ and let $n$ be some number. Consider random walk
$R$ starting from some point in $\partial$B(n) and let 
\[
T=\inf\{t:R(t)\not\in B(An)\text{ or }R(t)\in B(an)\}.
\]
Let $p_{n}=\mathbb{P}(R(T)\not\in B(An)$. Then 
\[
\lim_{n\to\infty}p_{n}=\frac{a^{-2}-1}{a^{-2}-A^{-2}},
\]
uniformly in the starting point.}

\medskip 

We do not need to bother with the case that $|R(\tau_{i})|\le M$
because then $\mathbb{P}(|R(\sigma)|>2|R(\tau_{i})|)=1$ and is certainly
bigger than any $p<\frac{4}{5}.$ We use lemma \ref{lem:exit annulus}
with $n_{\textrm{lemma \ref{lem:exit annulus}}}=|R(\tau_{i})|$, $A=2$
and some $a>\frac{1}{2}$ (e.g. 0.51). We take $M$ sufficiently large
so as to satisfy both that the probability is sufficiently close to
its limit in lemma \ref{lem:exit annulus} and so that the first time
that $|R(t)|<a|R(\tau_{i})|$ it still satisfies $|R(t)|\ge\frac{1}{2}|R(\tau_{i})|+1$.}
\ref{lem:exit annulus}, if $M$ is sufficiently large then there
exists $p>\frac{1}{2}$ such that $\mathbb{P}(|R(\sigma)|>2|R(\tau_{i})|)>p$
for all $i$, even after conditioning on $R(\tau_{i})$ (lemma \ref{lem:exit annulus}
allows to take any $p<\frac{4}{5}$). Fix $M$ to satisfy this property.
Since the probability that random walks stays in $B(r)$ for $\lambda r^{2}$
time is smaller than $Ce^{-c\lambda}$, uniformly in $r$ (lemma \ref{lem:exit expo}),
we see that $\mathbb{P}(\tau_{i+1}=\sigma_{i})>1-Ce^{-c\lambda}$.
Thus, if $\lambda$ is sufficiently large we have 
\[
\mathbb{P}(|R(\tau_{i+1})|>2|R(\tau_{i})|)>p'>\tfrac{1}{2}.
\]
Fix $\lambda$ to satisfy this property. We see that $\log_{2}|R(\tau_{i})|$
is a random walk with a drift. More formally, we can couple $\log_{2}|R(\tau_{i})|$
to the random walk $S$ on $\mathbb{Z}$ that walks from $k$ to $k+1$
with probability $p'$ and to $k-1$ with probability $1-p'$ so that
$\log_{2}|R(\tau_{i})|\ge S_{i}$ always. Standard estimates for one
dimensional random walk with a drift show that there exists some $\mu>0$
such that
\[
\mathbb{P}(\log|R(\tau_{i})|<\mu i)\le Ce^{-\mu i}.
\]
Using this with $i=\lceil2(\log n)/\mu\rceil$ gives\[
\mathbb{P}(\tau_{i}<n)\le
\mathbb{P}(|R(\tau_{i})|<n)
\HiddennoteLink{11cm}{\OptionalColour{\HNtemporaryhack{\le\;}{\le}}}{If $|R(\tau_{i})|<n$ then $\log|R(\tau_{i})|<\log n$ and since $i\ge2(\log n)/\mu$
we get 
\[
\log|R(\tau_{i})|<\mu i/2<\mu i
\]
Here the 2 is not needed. This is for the next inequality.} \mathbb{P}(\log|R(\tau_{i})|<\mu i)
\HiddennoteLink{11cm}{\OptionalColour{\HNtemporaryhack{\le\;}{\le}}}{Recall that all logs are to base 2. So 
\[
e^{-\mu i}\le e^{-2\log n}\le n^{2}
\]
(we do not need the 2 for this step, $2\log_{e}2$ would have been
enough, but 2 is shorter to write).} Ce^{-\mu i}\le Cn^{-2}.
\] On the other hand, between $\tau_{i}$ and $\tau_{i+1}$, the sum
of $(R(t)|+1)^{-2}$ is bounded by $\lambda\max\{M^{2},4\}$. The
lemma is thus proved with $K=C\lambda\max\{M^{2},4\}/\mu$. 
\end{proof}
The next lemma is the corresponding lower bound.
\begin{lem}
\label{lem:sraksraklower}There exists a $K$ with the following property.
Let $R$ be random walk on $\mathbb{Z}^{4}$ starting from $0$ and
let $n>0$. Then 
\[
\mathbb{P}\Big(\sum_{t=1}^{n}(|R(t)|+1)^{-2}<\frac{1}{K}\log n\Big)\le Cn^{-1/K}
\]
\end{lem}

\begin{proof}
The proof is similar to the previous lemma, but easier as we do not
need to control excursions of the random walk `inside'. Define therefore
stopping times $\tau_{i}$ by $\tau_{i}=\min\{t:|R(t)|>2^{i}\}$.
By lemma \ref{lem:exit time lower} we have that $\mathbb{P}(\tau_{i+1}>\tau_{i}+c4^{i})>c$,
where the constants are uniform in the value of $R(\tau_{i})$. Thus,
if we define $G_{i}$ to be the event that $\tau_{i+1}>\tau_{i}+c4^{i}$,
\HNOC{7cm}{then}{$c_{1}$ is a new constant, it was not defined
before. We claim that there exists some constant $c$ such that $\mathbb{P}(G_{i+1}\,|\,G_{1},\dotsc,G_{i})>c$
and named that constant $c_{1}$.} $\mathbb{P}(G_{i+1}\,|\,G_{1},\dotsc,G_{i})>c_{1}$.
We conclude that for every $k$, 
\[
\mathbb{P}(|\{i\le k:G_{i}\}|<(c_{1}/2)k)\le Ce^{-ck}.
\]
To move from stopping times to absolute times we note that, if we
define $k$ using $2^{k}\approx n^{1/4}$ then lemma \ref{lem:exit expo}
says that
\[
\mathbb{P}(\tau_{k}>n)\le Ce^{-c\sqrt{n}}.
\]
Since $G_{i}$ implies that $\sum_{t=\tau_{i}+1}^{\tau_{i+1}}(|R(t)|+1)^{-2}>c$,
the lemma is proved.
\end{proof}
\begin{lem}
\label{lem:Lawler}Let $R_{0},R_{1},\dotsc,R_{k}$ be independent
random walks in $\mathbb{Z}^{4}$ starting from $\partial B(n)$ and
stopped when hitting $\partial B(m)$ for some $m>n$. Assume further
that their starting points satisfy $|R_{0}(0)-R_{i}(0)|>\frac{n}{\log n}$.
Let $E$ be the event that $R_{0}\cap R_{i}\ne\emptyset$ for all
$i$. Then 
\[
\mathbb{P}\bigg(E\cap\Big\{\sum_{i=1}^{k}|R_{0}\cap R_{i}|<ck\log n\Big\}\bigg)<2e^{-ck}+Ckn^{-c}\log\Big(\frac{m\log n}{n}\Big)
\]
\end{lem}

\begin{proof}
Let $\lambda$ be some (large) parameter to be fixed later. Following
Lawler \cite[\S 3.4]{Law91}, we define a time $t$ as `good' if 
\begin{enumerate}
\item $\sum_{s=1}^{\sqrt{n}}|R_{0}(t+s)-R_{0}(t)|^{-2}>\frac{1}{\lambda}\log n$
\item $\sum_{s_{1},s_{2}=1}^{\sqrt{n}}|R_{0}(t+s_{1})-R_{0}(t)|^{-2}|R_{0}(t+s_{1}+s_{2})-R_{0}(t+s_{1})|^{-2}\le\lambda\log^{2}n$
\end{enumerate}
(the sums are up to $\sqrt{n}$ and not $n^{2}$ on purpose). For
the purpose of this definition let us extend $R_{0}$ to infinity
(so that a time very close to the hitting time of $\partial B(m)$
can still be good). We wish to show that (for a fixed $t$), the probability
that $t$ is good is high. To estimate the probability that the first
condition is violated we use lemma \ref{lem:sraksraklower} (with
$n_{\textrm{lemma \ref{lem:sraksraklower}}}=\sqrt{n}$) and get that,
if $\lambda$ is sufficiently large, the probability that the first
condition does not hold is $\le Cn^{-c}$.

The estimate for the second condition follows from lemma \ref{lem:sraksrak}.
Indeed, we fix $t$ and $s_{1}$ and use lemma \ref{lem:sraksrak}
for the random walk after time $t+s_{1}$ and get that the probability
that the sum over $s_{2}$ is larger than $C_{1}\log n$ is smaller
than $C/n$ for some $C_{1}$ sufficiently large. Summing over $s_{1}\in[t,t+\sqrt{n}]$
shows that with probability at least $1-C/\sqrt{n}$ this estimate
holds for all $s_{1}$, and then the sum over $s_{1}$ can be estimated
by a second application of lemma \ref{lem:sraksrak} after time $t$.
We get that the second condition holds with probability at least $1-C/\sqrt{n}$,
if only $\lambda\ge C_{1}^{2}$. We fix $\lambda$ to some value satisfying
both requirements and get that $t$ is good with probability $1-Cn^{-c}$.

With the fact that good times are typical established, we can proceed.
Let 
\begin{gather*}
\tau_{i}\coloneqq\{t:R_{i}(t)\in R_{0}\}\qquad\sigma_{i}\coloneqq\min\{t:R_{0}(t)=R_{i}(\tau_{i})\}\\
\rho_{i}\coloneqq\min\{t:R_{i}(t)\not\in B(m)\}.
\end{gather*}
Let $E_{i}$ be the event that $R_{i}\cap R_{0}\ne\emptyset$, that
$\sigma_{i}$ is good, and that $R_{0}(\sigma_{i})\in B(m-\sqrt{n})$.
Since $\tau_{i}$ is a stopping time, the portion of $R_{i}$ after
$\tau_{i}$ is simple random walk starting from $R_{0}(\sigma_{i})$,
and using the definition of a good time and lemma \ref{lem:heat kernel}
we have
\begin{align*}
\mathbb{E}(|R_{i}[\tau_{i},\rho_{i})\cap R_{0}[\sigma_{0},\sigma_{0}+n]|\,\big|\,E_{i}) & \ge c\log n\\
\mathbb{E}(|R_{i}[\tau_{i},\rho_{i})\cap R_{0}[\sigma_{0},\sigma_{0}+n]|^{2}\,|\,E_{i}) & \le C\log^{2}n
\end{align*}
hence there is some $c_{1}>0$ such that 
\[
\mathbb{P}(|R_{i}\cap R_{0}|>c_{1}\log n\,\big|\,E_{i})>c_{1}.
\]
Furthermore, this estimate for $i$ holds even if we condition on
all $j<i$, because of the conditioning on $\sigma_{i}$ being good.
We get from standard inequalities for i.i.d.\ random variables that
\[
\mathbb{P}\bigg(\sum_{i}|R_{i}\cap R_{0}|<ck\log n\,\Big|\,E_{1}\cap\dotsc\cap E_{k}\bigg)<2e^{-ck}.
\]
Thus the same estimate holds when we replace conditioning with intersection,
namely
\[
\mathbb{P}\bigg(E_{1}\cap\dotsc\cap E_{k}\cap\Big\{\sum_{i}|R_{i}\cap R_{0}|<ck\log n\Big\}\bigg)<2e^{-ck}.
\]
Thus we are left with the event $E\setminus(E_{1}\cap\dotsc\cap E_{k}),$
where $E$ is the event from the statement of the lemma. In other
words we need only estimate the probability that for some $i$ either
$\sigma_{i}$ is bad or $R_{0}(\sigma_{i})\not\in B(m-\sqrt{n})$.
The first of those can be estimated with a crude first moment estimate.
Indeed, for $\sigma_{i}$ bad we have
\[
\mathbb{P}(\sigma_{i}\text{ is bad})\le\sum_{t}\mathbb{P}(t\text{ is bad},R_{0}(t)\in R_{i})
\]
Now, the event that $t$ is bad does not depend on $R_{0}(t)$ since
it examines only the future, and $R_{0}$ was extended to infinity
for the purpose of the definition. Hence we can continue to write
\begin{align*}
 & =\sum_{t}\mathbb{P}(t\text{ is bad})\mathbb{P}(R_{0}(t)\in R_{1})\le Cn^{-c}\sum_{t}\mathbb{P}(R_{0}(t)\in R_{1})\\
 & \le Cn^{-c}\log(m\log n/n),
\end{align*}
where the last inequality is from lemma \ref{lem:intersection expectation}.
The other event, i.e.\ $R_{0}(\sigma_{i})\not\in B(m-\sqrt{n})$,
has much smaller probability. Indeed, by lemma \ref{lem:not near my boundary}
the probability that $|R_{0}\setminus B(m-\sqrt{n})|>n^{3/2}$ is
smaller than $\exp(-cn^{1/2})$, while any set of size $\le n^{3/2}$
has probability less than $Cm^{-2}n^{3/2}\le Cn^{-1/2}$ to be hit
by random walk (lemma \ref{lem:heat kernel}). Thus
\[
\mathbb{P}(R_{0}(\sigma_{i})\not\in B(m-\sqrt{n}))\le\exp(-cn^{1/2})+n^{-1/2},
\]
which is much smaller than what we need. Summing over $i$ proves
the lemma.
\end{proof}
\begin{lem}
\label{lem:hard to be hittable}Let $s_{1},s_{2}\in\partial B(n)$
with $|s_{1}-s_{2}|>\frac{n}{\log n}$. Let $r\in\mathbb{N}$ be some
parameter and let $\log^{-1}n\le\varepsilon<\frac{2}{3}$. Let $R_{i}$
be random walks starting from $s_{i}$ and stopped at $\partial B(m)$,
for $n\le m\le ne^{4\sqrt{\log n}}$. 
\begin{equation}
\mathbb{P}(\mathbb{P}(R_{1}\cap R_{2}\ne\emptyset\,|\,R_{1})>1-\varepsilon)\le\frac{K\varepsilon}{\log^{1/4}n}\label{eq:lem hard}
\end{equation}
for some constant $K$.
\end{lem}

\begin{proof}
Assume already at this step that $K\ge1$. We first note that for
$\varepsilon\in(c,\frac{2}{3})$ the lemma follows immediately from
the fact (lemma \ref{lem:intersection}) that 
\[
\mathbb{P}(R_{1}\cap R_{2}\ne\emptyset)\le\frac{\log((m\log n)/n)}{\log n}\le\frac{C}{\sqrt{\log n}}.
\]
Let therefore $c_{1}>0$ be some constant to be fixed later and assume
$\varepsilon<c_{1}$ (keeping in mind that $K$ might depend on $c_{1}$).
Hence we may also assume that $n$ is sufficiently large since for
small $n$ we will have $\log^{-1}n>c_{1}$ and the lemma follows
from the argument for $\varepsilon\ge c_{1}$. 

Denote the left-hand side of (\ref{eq:lem hard}) by $\delta$. Let
$S_{1},\dotsc,S_{k}$ be $k$ random walkers starting from $s_{2}$.
By assumption, the probability that they all intersect $R_{1}$ is
at least $\delta(1-\varepsilon)^{k}$. Denote this event by $E$.
We use lemma \ref{lem:Lawler} and get\begin{align}
\mathbb{P}\bigg(E\cap\Big\{\sum_{i=1}^{k}|R_{0}\cap S_{i}|<ck\log n\Big\}\bigg)
&<2e^{-ck}+\frac{Ck\log(m\log n/n)}{n^{c}}\nonumber\\
&\HNOC{7cm}{<}{$C_1$, $c_2$ and $c_3$ are all new constants that we just defined (any constants with the same name in other lemmas are unrelated)}
\;2e^{-c_{2}k}+\frac{C_{1}k}{n^{c_{3}}}\label{eq:from Lawler lem}
\end{align}where the second inequality is because $m\le ne^{4\sqrt{\log n}}$.
Let 
\[
k\coloneqq\Big\lceil\max\Big\{\frac{4\log\log n}{c_{2}},\frac{1}{\varepsilon^{1/6}}\Big\}\Big\rceil.
\]
We write 
\[
e^{-c_{2}k}\le e^{-c_{2}k/2}\cdot\frac{1}{\log^{2}n}\stackrel{(*)}{<}(1-\varepsilon)^{k}\cdot\frac{1}{\log^{2}n}
\]
where in $(*)$ we used (for the first term) that $\varepsilon<c_{1}$
for some $c_{1}$ sufficiently small. If $\delta\le K\varepsilon\log^{1/4}n$
then the lemma is proved. In the other case we get $\delta>K\varepsilon\log^{-1/4}n\ge K\log^{-5/4}n$
and hence 
\begin{equation}
2e^{-c_{2}k}<\tfrac{1}{4}\delta(1-\varepsilon)^{k},\label{eq:exp(-ck) is small}
\end{equation}
if only $K\ge8$. 

For the second term in (\ref{eq:from Lawler lem}), we write
\begin{equation}
(1-\varepsilon)^{k}=\min\{(1-\varepsilon)^{\lceil\varepsilon^{-1/6}\rceil},(1-\varepsilon)^{\lceil(4/c_{2})\log\log n\rceil}\}\ge c\log^{-1/4}n,\label{eq:(1-eps)^k}
\end{equation}
where in the last inequality we used that $\varepsilon<c_{1}$ for
some $c_{1}$ sufficiently small. Hence
\[
\tfrac{1}{4}\delta(1-\varepsilon)^{k}\ge\tfrac{1}{4}\frac{K}{\log^{5/4}n}\cdot c\log^{-1/4}n\ge c\log^{-3/2}n\ge\frac{C_{1}k}{n^{c_{3}}}
\]
where the penultimate inequality follows from our assumption that
$K\ge$1 and the last inequality holds for $n>n_{0}$ for some absolute
$n_{0}$ (we used here that $k\le\log^{1/6}n$ which follows from
its definition and $\varepsilon>\log^{-1}n$). Together with (\ref{eq:from Lawler lem})
and (\ref{eq:exp(-ck) is small}) we get
\[
\mathbb{P}\bigg(E\cap\Big\{\sum_{i=1}^{k}|R_{0}\cap R_{i}|<ck\log n\Big\}\bigg)<\tfrac{1}{2}\delta(1-\varepsilon)^{k}.
\]
Hence 
\[
\mathbb{P}\bigg(\sum_{i=1}^{k}|R_{0}\cap R_{i}|\ge ck\log n\bigg)>\tfrac{1}{2}\delta(1-\varepsilon)^{k}
\]
or
\begin{equation}
\mathbb{E}\bigg(\sum_{i=1}^{k}|R_{0}\cap R_{i}|\bigg)^{8}>ck^{8}\log^{8}n\cdot\delta(1-\varepsilon)^{k}.\label{eq:delta above}
\end{equation}
We now use lemma \ref{lem:the moments that killed me} with $r_{\textrm{\ref{lem:the moments that killed me}}}=8$.
We get
\begin{align*}
\lefteqn{\mathbb{E}\bigg(\sum_{i=1}^{k}|R_{0}\cap R_{i}|\bigg)^{8}}\qquad\\
 & \le Ck^{3/2}\log\Big(\frac{m\log n}{n}\Big)\Big(\log(m)+k^{2}\log\Big(\frac{m\log n}{n}\Big)\Big)^{7}\\
 & \le Ck^{3/2}\log^{15/2}n,
\end{align*}
where we used that $m\le ne^{4\sqrt{\log n}}$ and $k\le C\log^{1/6}n$.
With (\ref{eq:delta above}) we get 
\[
\delta<\frac{Ck^{3/2}\log^{15/2}n}{ck^{8}\log^{8}n(1-\varepsilon)^{k}}<\frac{C}{k^{6}(1-\varepsilon)^{k}\log^{1/2}n}<\frac{C\varepsilon}{\log^{1/4}n}.
\]
where in the last inequality we used (\ref{eq:(1-eps)^k}) and $k\ge\varepsilon^{-1/6}$.
The lemma is thus proved.
\end{proof}
\begin{rem}
The same proof can give that 
\[
\mathbb{P}(\mathbb{P}(R_{1}\cap R_{2}\ne\emptyset\,|\,R_{1})>1-\varepsilon)\le\frac{C(r)\varepsilon^{r}}{\log^{1/4}n}
\]
for any desirable $r$ (by increasing the power 8 used in (\ref{eq:delta above})
and elsewhere). The case of $r=1$ will suffice for the proof of the
theorem.
\end{rem}

\begin{lem}
\label{lem:one step coupling}Let $s_{1},s_{2}\in\partial B(n)$ with
$|s_{1}-s_{2}|>\frac{n}{\log n}$. Let $R_{i}$ be random walks starting
from $s_{i}$ and stopped at $\partial B(m)$, for $m\in[2n,ne^{4\sqrt{\log n}}]$.
Let $E_{i}$ be events defined for $R_{i}$. Then there exists a coupling
of $R_{i}$ which satisfies, with probability 
\[
1-C\Big(\mathbb{P}(E_{1})+\mathbb{P}(E_{2})+\frac{1}{\log n}\Big)
\]
that $E_{i}$ happened for $R_{i}$, that $R_{1}\cap R_{2}=\emptyset$
and that the end points of the two walkers $e_{i}$ satisfy $|e_{1}-e_{2}|>\frac{m}{\log m}$.
\end{lem}

\begin{proof}
We may assume $n$ is sufficiently large as otherwise the lemma can
be made to hold vacuously by taking $C$ sufficiently large. Similarly
we may assume $\mathbb{P}(E_{i})$ are sufficiently small. We may
assume without loss of generality that the events $E_{i}$ contain
the event that the time that $R_{i}$ reached $\partial B(m)$ is
at most some $T$ ($T$ is the same for $R_{1}$ and $R_{2}$). Indeed,
if this is not the case, add this condition, apply the lemma, and
then take $T\to\infty$ and take a weak subsequential limit of the
resulting couplings, in the space of measures on couples of paths.
Hence we will make this assumption from now on.

Let $A'_{i}$ be the set of paths of length $T$ starting from $s_{i}$
and such that $E_{i}$ holds. Denote $\mu\coloneqq\mathbb{P}(E_{1})+\mathbb{P}(E_{2})+1/\log n$.
Define 
\[
A_{i}=A_{i}'\setminus\Big\{\gamma:\mathbb{P}(R_{3-i}\cap\gamma\ne\emptyset)<1-3\mu\Big\}.
\]
(here and below intersections are only counted up to the first hitting
time of $\partial B(m)$ --- the rest of the path is ignored, for
both $\gamma$ and $R_{3-i}$). \HNOC{6cm}{By lemma }{Here we do
not need the $\log^{-1/4}$ factor in lemma \ref{lem:hard to be hittable}
--- we will need it later when we use this lemma again.}\ref{lem:hard to be hittable},
$|A_{i}|\ge2^{T}(1-C\mu)$. Assume for simplicity of notation that
$|A_{1}|\le|A_{2}|$.

We wish to use the marriage lemma \cite{marriage}. We therefore construct
a bipartite graph on $A_{1}\cup A_{2}$ such that each $\gamma\in A_{1}$
is connected to all $\delta\in A_{2}$ such that $\gamma\cap\delta=\emptyset$
and such that $|h(\gamma_{1})-h(\gamma_{2})|>m/\log m$, where $h(\gamma)$
is the point where the path first hits $\partial B(m)$, which must
exist for every $\gamma\in A_{1}\cup A_{2}$ by our assumption on
$T$. For a set $B\subseteq A_{1}$ we denote 
\[
\partial B\coloneqq\{\gamma\in A_{2}:\exists\delta\in B:\gamma\sim\delta\}
\]
To apply the marriage lemma to $A_{1}$ we need to show that for every
$B\subseteq A_{1}$, $|\partial B|\ge|B|$.

We first note that by definition, every $\gamma\in A_{1}$ satisfies
that $\mathbb{P}(R_{2}\cap\gamma\ne\emptyset)>3\mu$ so 
\[
\mathbb{P}(\{R_{2}\cap\gamma\ne\emptyset\}\setminus E_{2})>2\mu.
\]
Throwing in the ending condition gives
\[
\mathbb{P}\Big(\{R_{2}\cap\gamma\ne\emptyset\}\setminus\Big(E_{2}\cup\Big\{|h(\gamma)-h(R_{2})|\le\frac{m}{\log m}\Big\}\Big)\Big)>2\mu-\frac{C}{\log^{3}m}
\]
where we used $m\ge2n$ and lemma \ref{lem:hitting measure}. Further,
$2\mu-C/\log^{3}m>\mu$, since $\mu\ge1/\log n$ and $n$ is sufficiently
large. In $\partial$ notation, this gives $|\partial\{\gamma\}|>\mu2^{T}$.
Hence we can restrict our attentions to $B$ such that $|B|>\mu2^{T}$.
Denote $\varepsilon=|B|2^{-T}$. 

Apply lemma \ref{lem:hard to be hittable}, with  $\varepsilon_{\textrm{lemma \ref{lem:hard to be hittable}}}=\varepsilon+2\mu$.
Since lemma \ref{lem:hard to be hittable} requires $\varepsilon+2\mu<2/3$,
let us assume that $\mu<\frac{1}{6}$ so we can take any $B$ with
$|B|\le\frac{1}{2}2^{T}$. We get that 
\[
\mathbb{P}(\mathbb{P}(R_{1}\cap R_{2}\ne\emptyset\,|\,R_{1})>1-\varepsilon-2\mu)\le\frac{C(\varepsilon+2\mu)}{\log^{1/4}n}<\varepsilon
\]
for $n$ sufficiently large (recall that $\varepsilon>\mu$). Adding
$E_{2}$ and the endpoint condition, as above, shows that 
\[
\mathbb{P}(\mathbb{P}(\{R_{1}\cap R_{2}=\emptyset\}\setminus(E_{2}\cup\{|h(R_{1})-h(R_{2})|\le m/\log m\})\,|\,R_{1})<\varepsilon)<\varepsilon.
\]
In $\partial$ notation this means that $|\{\gamma:|\partial\{\gamma\}|<\varepsilon2^{T}\}|<\varepsilon2^{T}$.
But this means that there is at least one $\gamma\in B$ which does
not satisfy this condition, and hence $|\partial\{\gamma\}|\ge\varepsilon2^{T}=|B|$.
But this of course means that $|\partial B|\ge|\partial\{\gamma\}|\ge|B|$,
so in this case too the condition is satisfies.

We are left with the case $|B|>\frac{1}{2}2^{T}$. In this case we
turn our attention to $A_{2}$, and recall that any $\gamma\in A_{2}$
satisfies that $\mathbb{P}(\gamma\cap R_{1}=\emptyset)>3\mu$. Again
we add the condition $E_{1}$ and the end point condition and get
$|\partial\{\gamma\}|>\mu$. In particular every set $B$ with $|B|\ge|A_{1}|-\mu$
has every $\gamma$ in its boundary, i.e.~$|\partial B|=|A_{2}|\ge|A_{1}|\ge|B|$.
Finally, let $\varepsilon=(|A_{1}|-|B|)/2^{T}$ and note that the
only remaining case is $\frac{1}{2}>\varepsilon>\mu$. We apply lemma
\ref{lem:hard to be hittable}, this time with $R_{1}$ and $R_{2}$
exchanged, and get
\[
\mathbb{P}(\mathbb{P}(R_{2}\cap R_{1}\ne\emptyset)\,|\,R_{2})>1-\varepsilon-2\mu)\le\frac{C(\varepsilon+2\mu)}{\log^{1/4}n}<\varepsilon
\]
for $n$ sufficiently large. Again we conclude that $|\{\gamma\in A_{2}:|\partial\{\gamma\}|<\varepsilon2^{T}\}|<\varepsilon2^{T}$.
Every $\gamma$ which satisfy $|\partial\{\gamma\}|\ge\varepsilon2^{T}$
must also satisfy $\partial\{\gamma\}\cap B\ne\emptyset$ or equivalently
$\gamma\in\partial B$. We get 
\begin{align*}
|\partial B| & \ge|\{\gamma\in A_{2}:|\partial\{\gamma\}|\ge\varepsilon2^{T}\}|=|A_{2}|-|\{\gamma\in A_{2}:|\partial\{\gamma\}|<\varepsilon2^{T}\}|\\
 & \ge|A_{1}|-\varepsilon2^{T}=|B|,
\end{align*}
as needed. This finishes the requirements of the marriage lemma and
we get that there exists some one-to-one map $\varphi:A_{1}\to A_{2}$
such that $\gamma\sim\varphi(\gamma)$.

The coupling is derived from $\varphi$ by chopping the parts of the
walk after the hitting time of $\partial B(m)$. Indeed, a coupling
is simply a measure on the space of couples of paths, and we take
this measure to be $\mu+\nu$ where
\[
\mu(\gamma,\delta)=\frac{1}{2^{T}}|\{(\gamma',\delta')\in A_{1}\times A_{2}:\gamma=\gamma'[0,\tau(\gamma')],\delta=\delta'[0,\tau(\delta')],\delta'=\varphi(\gamma')\}|
\]
where $\tau(\gamma')$ is the time $\gamma'$ hit $\partial B(m)$;
and where $\nu$ is an arbitrary coupling of $A_{1}^{c}$ with $\varphi(A_{1})^{c}$
(say i.i.d.). The lemma is proved.
\end{proof}
\begin{lem}
\label{lem:forEi}There exists some constant $C_{1}$ such that the
following holds. Let $x\in B(n)$ and let $k<n$. Let $R$ be random
walk starting from $x$ and stopped on $\partial B(n),$and let $y$
be its hitting point. For a $z\in\partial B(n)$ denote by $R_{2}$
a second random walk started from $z$ and continuing till infinity.
Let $H=H_{n,k}$ be the event
\begin{multline*}
H=\Big\{\exists z\in\partial B(n):|z-y|>k,\\
\mathbb{P}(R_{2}\cap R\ne\emptyset\,|\,R)>C_{1}\frac{(\log\log n)\log(n\log^{3}n)/k}{\log n}\Big\}
\end{multline*}
(this is an event on $R$). Then
\[
\mathbb{P}(H)<\frac{C}{\log^{2}n}.
\]
\end{lem}

(this lemma will be used with $k=n/\log n$, so the reader may think
about this case only).
\begin{proof}
We may assume that $k>n^{3/4}$ as otherwise the lemma holds trivially
if we take $C_{1}\ge4$. Similarly we may assume $n$ is sufficiently
large.

Let $K$ be some parameter to be fixed later. Let $\gamma$ be a path
starting from $x$ and stopped at $\partial B(n)$, and let $\ell$
be its length. Let $t\in\{1,\dotsc,\ell-n\}$. As in the proof of
lemma \ref{lem:Lawler}, we say that $t$ is a good time for $\gamma$
if the following two conditions hold.
\begin{enumerate}
\item [(1)]$\sum_{s=1}^{n}|\gamma(t+s)-\gamma(t)|^{-2}>\frac{1}{K}\log n$
\item [(2)]$\sum_{s_{1},s_{2}=1}^{n}|\gamma(t+s_{1})-\gamma(t)|^{-2}|\gamma(t+s_{1}+s_{2})-\gamma(s+s_{1})|^{-2}\le K\log^{2}n$
\end{enumerate}
We say that $\gamma$ itself is good if the following conditions
hold
\begin{enumerate}
\item \label{enu:no recoil}Let $i\in\{1,\dotsc,\lceil\log_{2}n\rceil\}$,
and let $\tau_{i}$ be the first time the path $\gamma$ exited $B(n-2^{i})$.
Then we require that the portion of $\gamma$ after $\tau_{i}$ does
not leave the ball $B(\gamma(\tau_{i}),2^{i}\log^{3}n).$
\item \label{enu:weight}For every $i\in\{1,\dotsc,\lceil\log_{2}n\rceil\}$
we want $|\gamma\cap(B(n)\setminus B(n-2^{i}))|\le K4^{i}\log\log n$.
\item \label{enu:bad}If $\gamma_{b}=\{\gamma(t):t\text{ bad}\}$ then $\max\{\mathbb{P}(R_{2}\cap\gamma_{b}\ne\emptyset):z\textrm{ s.t. }|z-y|>k\}\le1/\log^{3}n$,
where $z$, $R_{2}$, $y$ and $k$ are as in the statement of the
lemma.
\end{enumerate}
Our plan is to show that every good $\gamma$ belongs to $H$, and
then to estimate the probability that $\gamma$ is good. 

We start with showing that $\gamma\in H$. Fix therefore one good
$\gamma$ and one $z\in\partial B(n)$ with $|z-y|>k$ (as in the
definition of $H$, $y$ is the end point of $\gamma$). Let $\gamma_{b}$
be the set of $\gamma(t)$ for all bad $t$ and let $\gamma_{g}\coloneqq\gamma\setminus\gamma_{b}$.
We write
\[
\mathbb{P}(R_{2}\cap\gamma\ne\emptyset)\le\mathbb{P}(R_{2}\cap\gamma_{b}\ne\emptyset)+\mathbb{P}(R_{2}\cap\gamma_{g}\ne\emptyset)=I+II.
\]
(Recall that $R_{2}$ is a random walk starting from $z$ and continuing
till infinity). By property \ref{enu:bad}, $I\le1/\log^{3}n$. 

To bound $II$ we write, as in lemma \ref{lem:Lawler},
\begin{equation}
II\le\frac{\mathbb{E}(|R_{2}\cap\gamma|)}{\mathbb{E}(|R_{2}\cap\gamma|\,\big|\,R_{2}\cap\gamma_{g}\ne\emptyset)}\label{eq:num denum}
\end{equation}
and estimate both terms. To estimate the numerator examine one $i\in\{1,\dotsc,\lceil\log_{2}n\rceil\}$
and let $X=\gamma\cap(B(n-2^{i-1})\setminus B(n-2^{i}))$. By \ref{enu:no recoil}
we know that $X\subset B(\gamma(\tau_{i}),2^{i}\log^{3}n)$ and further,
that $y\in B(\gamma(\tau_{i}),2^{i}\log^{3}n)$, so the fact that
$|z-y|>k$ implies that $\dist(z,X)\ge\max\{k-2^{i}\log^{3}n,2^{i-1}\}$.
By \ref{enu:weight} we know that $|X|\le K4^{i}\log\log n$. Bounding
the probability simply by the expected size of the intersection (using
lemma \ref{lem:heat kernel}) we get
\[
\mathbb{P}(R_{2}\cap X_{i}\ne\emptyset)\le\frac{K4^{i}\log\log n}{\max\{k-2^{i}\log^{3}n,2^{i-1}\}^{2}}.
\]
summing over $i$ gives 
\[
\mathbb{E}(|R_{2}\cap\gamma|)\le\sum_{i=1}^{\log n}\frac{K4^{i}\log\log n}{\max\{k-2^{i}\log^{3}n,2^{i-1}\}^{2}}\le CK(\log\log n)\log\frac{n\log^{3}n}{k}
\]
To estimate the denominator in (\ref{eq:num denum}) we notice, as
in the proof of lemma \ref{lem:Lawler}, that if you hit $\gamma$
in a good point then you have, always, probability at least some $c>0$
to hit $\gamma$ in $c\log n$ points afterwards. Hence the denominator
is $\ge c\log n$. Thus we get $II\le CK(\log\log n)(\log((n\log^{3}n)/k))/\log n$.
Together with the estimate of $I$ we get that, indeed, $\gamma\in H$,
if only $C_{1}$ is chosen sufficiently large (depending on $K$,
which has not yet been fixed).

We move to showing that the probability that a random walk from $x$
to $\partial B(n)$ is good is high. Denote the random walk by $R$.
For clause \ref{enu:no recoil} we fix one $i$. Lemma \ref{lem:escape curved boundary}
shows that the probability that $R$ exits $B(\tau_{i},2^{i}\log^{3}n)$
($\tau_{i}$ as in \ref{enu:no recoil}) before exiting $\partial B(n)$
is $\le C/\log^{3}n$. Summing over $i$ gives that the probability
to not satisfy \ref{enu:no recoil} is $\le C/\log^{2}n$. 

For clause \ref{enu:weight} let us again fix $i$. By lemma \ref{lem:exit annulus},
\begin{equation}
\mathbb{P}(|R\cap(B(n)\setminus B(n-2^{i}))|>\lambda4^{i})\le Ce^{-c\lambda}.\label{eq:not near my boundary}
\end{equation}
Using this with $\lambda=C_{1}\log\log n$ for some $C_{1}$ sufficiently
large and summing over $i$ we get that the probability to not satisfy
\ref{enu:weight} is $\le C/\log^{2}n$.

Finally, for clause \ref{enu:bad}, we can use a relatively rough
bound. We first note that for any $\gamma$, 
\begin{multline}
\max\{\mathbb{P}(R_{2}\cap\gamma_{b}\ne\emptyset):z\textrm{ s.t. }|z-y|>k\}\\
\le\sum_{i=1}^{\log_{2}n}C4^{-i}|\gamma_{b}\cap(B(n-2^{i})\setminus B(n-2^{i+1}))|.\label{eq:bad is just sizes}
\end{multline}
To apply this for $\gamma$ being a random walk, it will be convenient
to extend $R$ to infinity and say that a time $t$ is `$\text{bad}^{\infty}$'
if it satisfies conditions (1)-(2) with respect to this extension.
Thus any bad $t$ is also bad$^{\infty}$ but near the hitting time
of $\partial B(n)$ there might be bad$^{\infty}$ times which are
not bad. We denote $R_{\infty}=\{R(t):t\text{ bad}^{\infty}\}$ and
our next goal would be to estimate the size of $R_{\infty}\cap\text{the same annuli}$.

Let us first note that for any $t$, $\mathbb{P}(t\text{ is bad}^{\infty})\le Cn^{-c}$.
This was proved in the first two paragraphs of the proof of lemma
\ref{lem:Lawler} so let us not repeat the argument. Further, for
any $t$ we have that the event $t<\ell$ and the variable $R(t)$
are both independent of the event `$t$ is bad$^{\infty}$'. This
holds because $t<\ell$ examines only $R[0,t]$ while being bad examines
only $R(t,t+n)-R(t)$. Thus we get
\begin{multline*}
\mathbb{E}|R_{\infty}\cap(B(2^{n}-i^{i})\setminus B(2^{n}-2^{i+1})|\\
\le\sum_{t=0}^{\infty}\mathbb{P}(t\le\ell,R(t)\in B(2^{n}-2^{i})\setminus B(2^{n}-2^{i+1})\cdot\mathbb{P}(t\text{ bad}^{\infty})\le Cn^{-c}\cdot C4^{i},
\end{multline*}
where the last inequality follows from (\ref{eq:not near my boundary}).
Thus,
\[
\mathbb{E}\Big(\sum_{i=1}^{\log_{2}n}4^{-i}\mathbb{E}|R_{\infty}\cap(B(2^{n}-i^{i})\setminus B(2^{n}-2^{i+1}))|\Big)\le Cn^{-c}.
\]
Markov's inequality now tells us that
\[
\mathbb{P}\Big(\sum_{i=1}^{\log_{2}n}4^{-i}\mathbb{E}|R_{\infty}\cap(B(2^{n}-i^{i})\setminus B(2^{n}-2^{i+1}))|>n^{-c}\Big)\le Cn^{-c}
\]
Together with (\ref{eq:bad is just sizes}) and the fact that $R_{b}\subseteq R_{\infty}$
we get
\[
\mathbb{P}(\max\{\mathbb{P}(R_{2}\cap R_{b}\ne\emptyset):z\textrm{ s.t. }|z-y|>k\}>Cn^{-c}\,|\,R)\le Cn^{-c}.
\]
This was the last estimate. We get that $R$ is good with probability
$\ge C/\log^{2}n$, and the lemma is proved.
\end{proof}

\begin{proof}
[Proof of the theorem]Let $n_{1}$ be some constant to be fixed later.
We will construct the the coupling in steps, where at the $n^{\textrm{th}}$
step we have, with probability $p_{n}$, that for all $m\in[n_{1},n]$,
\begin{enumerate}
\item $H_{2^{m^{2}},2^{m^{2}}/m^{2}}$ does not hold for $R_{i}$;
\item $|R_{1}(\tau_{m,1})-R_{2}(\tau_{m,2})|>2^{m^{2}}/m^{2}$; and
\item $R_{1}[0,\tau_{m,1}]\cap R_{2}[0,\tau_{m,2}]=\emptyset$.
\end{enumerate}
Here $H$ is from lemma \ref{lem:forEi} and $\tau_{m,i}$ is the
time when $R_{i}$ hit $\partial B(2^{m^{2}})$. And, of course, we
will show that $p_{n}$ stays bounded below. For the first step take
 two independent random walks from $0$ and from $x$ to $\partial B(2^{n_{1}^{2}})$.
By a theorem of Lawler's (\cite[theorem 4.4.1]{Law91} for the case
that $x$ and 0 are neighbours), they have probability $\gtrsim1/n_{1}$
to not intersect. By lemma \ref{lem:forEi} the probability that they
satisfy $H_{2^{n_{1}^{2}}/n_{1}^{2}}$ is $\approx1/n_{1}^{4}$, and
the probability that they end up in points with distance at least
$2^{n_{1}^{2}}/n_{1}^{2}$ it $\approx1/n_{1}^{6}$ (lemma \ref{lem:hitting measure}).
Hence if $n_{1}$ is sufficiently large the probability that all 3
conditions are met is $\ge c/n_{1}-C/n_{1}^{4}-Cn_{1}^{6}$ and is
positive. This establishes the first step.

Assume now by induction that the coupling has been defined up to $n$.
To define the next step we apply lemma \ref{lem:one step coupling}.
We use it with $n_{\textrm{lemma \ref{lem:one step coupling}}}=2^{n^{2}}$,
$m_{\textrm{lemma \ref{lem:one step coupling}}}=2^{(n+1)^{2}}$ and
with the $s_{i}=R_{i}(\tau_{n,i})$. By definition the requirements
of lemma \ref{lem:one step coupling} from its $n$, $m$ and $s_{i}$
are satisfied.

As for the event $E_{i}$ of lemma \ref{lem:one step coupling}, we
define 
\[
E_{i}=\{R_{i}[\tau_{n,i},\tau_{n+1,i}]\cap R_{3-i}[0,\tau_{n,3-i}]=\emptyset\}\cup\{H_{2^{(n+1)^{2}},2^{(n+1)^{2}}/(n+1)^{2}}(R_{i})\}
\]
(the notation $H(R_{i})$ means that the event holds for $R_{i}$).
Note that the first condition involves the `other' walk, but only
its past. So the coupling we will get from the lemma also depends
on the past of the walkers. Lemma \ref{lem:one step coupling} will
now give us all 3 requirements for the next step of the induction,
so we only need to estimate $p_{n+1}$, which will be $p_{n}\cdot(\text{what comes from lemma \ref{lem:one step coupling}})$.

To estimate the probability of the $E_{i}$ we note that the first
condition ($R_{i}(\textrm{present})\cap R_{3-i}(\textrm{past})=\emptyset)$
is, by our inductive assumptions 1 and 2, at least 
\[
C_{1}\frac{(\log\log(2^{n^{2}}))\log(2^{n^{2}}n^{6}/(2^{n^{2}}/n^{2}))}{\log(2^{n^{2}})}\approx\frac{\log^{2}n}{n^{2}}.
\]
As for the second condition, by lemma \ref{lem:forEi} we have 
\[
\mathbb{P}(H_{2^{(n+1)^{2}},2^{(n+1)^{2}}/(n+1)^{2}})\le\frac{1}{\log^{2}\big(2^{(n+1)^{2}}/(n+1)^{2}\big)}\approx\frac{1}{n^{4}}.
\]
Thus $\mathbb{P}(E_{i})\approx(\log n)^{2}/n^{2}$ and plugging this
into lemma \ref{lem:one step coupling} we get that the coupling at
the $n^{\textrm{th}}$ step succeeds with probability 
\[
1-C\left(\mathbb{P}(E_{1})+\mathbb{P}(E_{2})+\frac{1}{\log2^{n^{2}}}\right)=1-O\left(\frac{\log^{2}n}{n^{2}}\right)
\]
or 
\[
p_{n+1}\ge p_{n}\left(1-C\frac{\log^{2}n}{n^{2}}\right).
\]
Fix $n_{1}$ so that $(C\log^{2}n)/n^{2}<1$ for all $n\ge n_{1}$
and the theorem is proved. 
\end{proof}

\appendix

\section{Random walk preliminaries}

\renewcommand{\thethm}{{A.\arabic{thm}}} 
\renewcommand{\theHthm}{{A.\arabic{thm}}}
\setcounter{thm}{0}
\begin{lem}
\label{lem:heat kernel}The probability $G(x,y)$ that random walk
starting from $x$ hits $y\ne x$ is $\approx|y-x|^{-2}$.
\end{lem}

\begin{proof}
See \cite[Theorem 1.5.4]{Law91}.
\end{proof}
\begin{lem}
\label{lem:Lawler's J}Let $R_{1}$ and $R_{2}$ be random walks starting
from $x$ and $0$, respectively. Let $n>|x|^{2}$. Then 
\[
\mathbb{E}(|R_{1}[0,n]\cap R_{2}[0,\infty)|)\approx\log(n/|x|^{2}).
\]
\end{lem}

\begin{proof}
Write 
\begin{multline*}
\mathbb{E}(|R_{1}[0,n]\cap R_{2}[0,\infty)|)=\sum_{i=0}^{n}\sum_{j=0}^{\infty}\mathbb{P}(R_{1}(i)=R_{2}(j))\\
\stackrel{(*)}{\approx}C\sum_{i=0}^{n}\sum_{j=0}^{\infty}\frac{1}{(i+j+1)^{2}}e^{-c|x|^{2}/(i+j+1)}\approx\log(n/|x|^{2}),
\end{multline*}
where $(*)$ comes from the fact that asking whether $R_{1}(i)=R_{2}(j)$
is the same as asking whether a single random walk of length $i+j$
starting from $0$ reaches $x$; and then the local central limit
theorem \cite[theorem 1.2.1]{Law91}.
\end{proof}
We also need a version of the last lemma with a stopping times replacing
the fixed time $n$.
\begin{lem}
\label{lem:intersection expectation}Let $R_{1}$ and $R_{2}$ be
random walks starting from $x$ and $0$, respectively and assume
$R_{1}$ is stopped when exiting a ball of radius $n$, for some $n>|x|$.
Then $\mathbb{E}(|R_{1}\cap R_{2}|)\le C\log(n/|x|^{2})$
\end{lem}

\begin{proof}
Write 
\[
\mathbb{E}(|R_{1}\cap R_{2}|)=\sum_{x\in B(n)}\mathbb{P}(x\in R_{1})\mathbb{P}(x\in R_{2})
\]
and estimate both terms using lemma \ref{lem:heat kernel}.
\end{proof}
\begin{lem}
\label{lem:exit expo}The probability that random walk $R$ has not
exited $B(n)$ by time $t$ is $\le Ce^{-ct/n^{2}}$.
\end{lem}

\begin{proof}
For any $x\in B(n)$, the local central limit theorem shows that random
walk exits $B(n)$ by time $n^{2}$ with probability $\ge c$ for
some $c>0$. Since this holds for all $x\in B(n)$ uniformly, it holds
that $R$ exits $B(n)$ in any time interval $[s,s+n^{2}]$ with probability
$>c$ independently of the past. The lemma follows.
\end{proof}
\begin{lem}
\label{lem:exit time lower}Let $R$ be a random walk starting from
an $x\in B(n)$ and let $\tau$ be its exit time from $B(2n)$. Then
$\mathbb{E}\tau\approx n^{2}$.
\end{lem}

\begin{proof}
The upper bound is given by lemma \ref{lem:exit expo} so we focus
on the lower bound. By the reflection principle \cite[Theorem 5.2.7]{Dur},
\[
\mathbb{P}(\max_{t\le T}|R(t)-x|>M)\le2\mathbb{P}(|R(T)-x|>M)
\]
for all $T$, $x$ and $M$ (the result is stated in \cite{Dur} in
the one-dimensional case, but the proof works equally in all dimensions).
Using this with $T=n^{2}$, $M=n-1$ and bounding $\mathbb{P}(|R(T)-x|>M)\le1-c$
for some $c>0$ independent of $n$ using the central limit theorem
proved the lemma. 
\end{proof}
\begin{lem}
\label{lem:exit annulus}Let $a<1<A$ and let $n$ be some number.
Consider random walk $R$ starting from some point in $\partial$B(n)
and let 
\[
T=\inf\{t:R(t)\not\in B(An)\text{ or }R(t)\in B(an)\}.
\]
Let $p_{n}=\mathbb{P}(R(T)\not\in B(An)$. Then 
\[
\lim_{n\to\infty}p_{n}=\frac{a^{-2}-1}{a^{-2}-A^{-2}},
\]
uniformly in the starting point.
\end{lem}

\begin{proof}
By \cite[Theorem 1.5.4]{Law91} the Green function of $\mathbb{Z}^{4}$
satisfies $G(0,x)=(\alpha+o(1))|x|^{-2}$ for some constant $\alpha>0$
(the $o(1)$ is as $|x|\to\infty$). Since the Green function is harmonic,
$G(R(t))$ is a martingale and hence
\begin{align*}
(\alpha+o(1)n^{-2} & =G(0,x)=\mathbb{E}(G(R(T)))\\
 & =p_{n}(A\alpha+o(1))n^{-2}+(1-p_{n})(a\alpha+o(1))n^{-2},
\end{align*}
where $x$ is the starting point of the walk. Moving terms around
proves the lemma.
\end{proof}
\begin{lem}
\label{lem:not near my boundary}Let $R$ be a random walk starting
from an $x\in B(n)$ and stopped when exiting $B(n)$. Let $k<n$.
Then 
\[
\mathbb{P}(|R\setminus B(n-k)|>\lambda k^{2})\le2e^{-c\lambda}.
\]
\end{lem}

\begin{proof}
Let $\tau_{i}$ be stopping times defined by $\tau_{0}=-\infty$ and
\[
\tau_{i+1}=\inf\{t\ge\tau_{i}+k^{2}:R(t)\not\in B(n-k)\}.
\]
Between $\tau_{i}$ and $\tau_{i+1}$ the random walk $R$ cannot
spend more than $k^{2}$ times in the complement of $B(n-k)$. On
the other hand, the central limit theorem shows that $\frac{1}{k}$($R(\tau_{i}+k^{2})-R(\tau_{i}))$
converges to a 4 dimensional gaussian variable. Since this has probability
$c$ to be outside any pre-given half space, we see that, for $k$
sufficiently large, 
\begin{equation}
\mathbb{P}(R(\tau_{i}+k^{2})\not\in B(n))>c\qquad\forall i\ge1\label{eq:walken raus}
\end{equation}
uniformly in the starting point (which must be outside $B(n-k)$,
by definition). A simple check shows that, in fact, (\ref{eq:walken raus})
holds for any $k\ge1$ (and not just for $k$ sufficiently large),
with perhaps a smaller value of $c$. Hence the walk is stopped between
any $\tau_{i}$ and $\tau_{i+1}$, with probability at least $c>0$,
independently of the past. The lemma is proved.
\end{proof}
\begin{lem}
\label{lem:Lawler annoys}Let $R_{1}$ be a random walk started from
some $x$ and stopped when exiting $B(2|x|)$. Let $R_{2}$ be a second
random walk started from 0 and continued without end. Then
\[
\mathbb{P}(R_{1}\cap R_{2}\ne\emptyset)\le\frac{C}{\log|x|}.
\]
\end{lem}

\begin{proof}
For $i\in\mathbb{N}$ let $B_{i}$ be the event that $R_{1}$ did
not exit $B(2|x|)$ before time $i|x|^{2}$ and that $|R_{1}(i|x|^{2})|<|x|/\log|x|$.
The central limit theorem tells us that $R_{1}$ has probability at
least $c>0$ to exit $B(2|x|)$ between $i|x|^{2}$ and $(i+1)|x|^{2}$,
uniformly in the starting position, while the local central limit
theorem \cite[theorem 1.2.1]{Law91} says that the probability that
$R_{1}(i|x|^{2})\in B(|x|/\log|x|)$ is bounded above by $C\log^{-4}|x|$,
again, uniformly in $R_{1}((i-1)|x|^{2})$. Hence $\mathbb{P}(B_{i})\le Ce^{-ci}\log^{-4}|x|$.

Let now $E_{i}$ (this time we allow also $i=0$) be the event that
$R_{1}$ did not exit $B(2|x|)$ before time $i|x|^{2}$ and that
$R_{1}[i|x|^{2},(i+1)|x|^{2})\cap R_{2}\ne\emptyset$. By \cite[theorem 4.3.3]{Law91},
$\mathbb{P}(E_{0})\le C/\log n$. For $i>0$, we note that the same
calculation as above shows that the probability that $R_{1}$ did
not exit $B(2|x|)$ before time $i|x|^{2}$ and that $R_{1}(i|x|^{2})\in B(2^{-m}|x|)\le Ce^{-ci}2^{-4m}$.
Under this condition, \cite[theorem 4.3.3]{Law91} tells us that the
probability that $R_{1}[i|x|^{2},(i+1)|x|^{2})$ intersects $R_{2}$
is $\le C(\log m)/\log|x|$ (we used here lemma \ref{lem:Lawler's J}
to calculate the term $J_{n}(x)$ that appears in \cite{Law91}).
Hence
\[
\mathbb{P}(E_{i}\setminus B_{i})\le\sum_{m=1}^{\infty}Ce^{-ci}2^{-4m}\frac{\log m}{\log|x|}\le\frac{Ce^{-ci}}{\log|x|}.
\]
Summing over $i$ proves the lemma.
\end{proof}
\begin{lem}
\label{lem:intersection}Let $R_{1}$ and $R_{2}$ be two independent
random walks starting from $s_{1}$ and $s_{2}$ respectively and
stopped when exiting $B(m)$ for some $m$. Then
\[
P(R_{1}\cap R_{2}\ne\emptyset)\le C\frac{\log(m/(|s_{1}-s_{2}|+2))}{\log(|s_{1}-s_{2}|+2)}.
\]
\end{lem}

\begin{proof}
The lemma holds trivially for $m<4|s_{1}-s_{2}|$ so assume this is
not the case. We extend $R_{2}$ to infinity (denote this by $R_{2}^{\infty}$)
and bound $\mathbb{P}(R_{1}\cap R_{2}\ne\emptyset)\le\mathbb{P}(R_{1}\cap R_{2}^{\infty}\ne\emptyset)$.
Then the lemma follows by repeatedly applying lemma \ref{lem:Lawler annoys}.
\end{proof}
\begin{lem}
\label{lem:hitting measure}Let $R$ be random walk starting from
$x\in B(n)$ and let $t$ be the time it exits $B(m)$ for some $m\ge2n$.
Then for every $y\in\partial B(m)$ we have $\mathbb{P}(R(t)=y)\le Cm^{-3}$.
\end{lem}

\begin{proof}
See \cite[lemma 6.7.3]{LL}. The lemma in \cite{LL} is formulated
for $m\ge4n$ but the Harnack inequality \cite[theorem 6.3.9]{LL}
says that this is the same, since the probability to hit $\partial B(m)$
at $y$ is a positive harmonic function in the starting point of the
walk in $B(m)$.
\end{proof}
\begin{lem}
\label{lem:escape curved boundary}Let $R$ be a random walk starting
from an $x\in B(n)$ and stopped when exiting $B(n)$. Then the probability
that $R$ exited $B(x,k)$ is $\le C(n-|x|)/k$.
\end{lem}

\begin{proof}
Strangely, even though this fact is classic, the only published proof
we are aware of of it is in \cite[theorem 2 (pg. 73)]{K07}. Readers
who know a reference with less `baggage' attached to it are welcome
to contact us. Anyway, the theorem in \cite{K07} is formulated for
any `isotropic graph', but $\mathbb{Z}^{4}$ is such a graph. The
result in \cite{K07} is for the probability of exiting $B(x,k)$
before exiting a half space, but using this with the half space tangent
to $B(n)$ at $xn/|x|$ will give an upper bound, since $B(n)$ is
a subset of this tangent half space. Finally, the result in \cite{K07}
is formulated only for the case that $n-|x|>C$ for some constant
$C$, but in our setting ($\mathbb{Z}^{4}$ and only an upper bound),
the result for any $|x|$ follows from it easily (or, alternatively,
follows with the exact same proof as in \cite{K07}).
\end{proof}


\begin{thebibliography}{1}
\bibitem{AHMWW}Omer Angel, Alexander E. Holroyd, James Martin, David
B. Wilson and Peter Winkler, \afs{https://doi.org/10.1214/ECP.v18-2275}{Avoidance
coupling}. Electron. Commun. Probab. 18 (2013), no. 58, 13 pp. 

\bibitem{BGGS}Itai Benjamini, Ori Gurel-Gurevich and Oded Schramm,
\afs{https://doi.org/10.1214/10-AOP569}{Cutpoints and resistance
of random walk paths}. Ann. Probab. 39:3 (2011), 1122--1136. 

\bibitem{Dur}Rick Durrett, \afs{https://doi.org/10.1017/9781108591034}{Probability---theory
and examples}. Fifth edition. Cambridge Series in Statistical and
Probabilistic Mathematics, 49. Cambridge University Press, Cambridge,
2019. 

\bibitem{K07}Gady Kozma, \afs{https://doi.org/10.1007/s11511-007-0018-8}{The scaling limit of loop-erased random walk in three dimensions}.
Acta Math. 199 (2007), no. 1, 29--152. 

\bibitem{Law91}Gregory F. Lawler, \afs{https://link.springer.com/book/10.1007/978-1-4614-5972-9}{Intersections of random walks}.
Probability and its Applications. Birkhäuser Boston, Inc., Boston,
MA, 1991.

\bibitem{LL}Gregory F. Lawler and Vlada Limic, \afs{https://doi.org/10.1017/CBO9780511750854}{Random walk: a modern introduction}.
Cambridge Studies in Advanced Mathematics, 123. Cambridge University
Press, Cambridge, 2010. 

\bibitem{LP}David A. Levin and Yuval Peres, \afs{https://doi.org/10.1090/mbk/107}{Markov chains and mixing times}.
Second edition. With contributions by Elizabeth L. Wilmer. With a
chapter on \textquotedbl Coupling from the past'' by James G. Propp
and David B. Wilson. American Mathematical Society, Providence, RI,
2017.

\bibitem{Lind}Torgny Lindvall, \emph{Lectures on the coupling method}.
Corrected reprint of the 1992 original. Dover Publications, Inc.,
Mineola, NY, 2002.

\bibitem{marriage}various authors, \afs{https://en.wikipedia.org/w/index.php?title=Hall
Wikipedia, retrieved 6/4/2024.
\end{thebibliography}
\end{document}